\documentclass[12pt]{amsart}
\usepackage{amssymb}
\usepackage{enumerate}
\usepackage{graphicx}
\usepackage[all,cmtip]{xy}
 \usepackage{tikz-cd}
 \usepackage{orcidlink}
\makeatletter
\@namedef{subjclassname@2010}{%
  \textup{2010} Mathematics Subject Classification}
\makeatother
\newtheorem{thm}{Theorem}[section]
\newtheorem{corollary}[thm]{Corollary}
\newtheorem{lemma}[thm]{Lemma}
\newtheorem{proposition}[thm]{Proposition}
\theoremstyle{definition}
\newtheorem{definition}[thm]{Definition}
\newtheorem{remark}[thm]{Remark}
\newtheorem{example}[thm]{Example}
\newtheorem{question}[thm]{Question}
\begin{document}
\baselineskip=15pt
\title{Uniformly $S$-essential submodules and uniformly $S$-injective uniformly $S$-envelopes}
\author[M. Adarbeh]{Mohammad Adarbeh $^{(\star)}$}
\address{Department of Mathematics, Birzeit University, Birzeit,  Palestine}
\email{madarbeh@birzeit.edu}
\author[M. Saleh]{Mohammad Saleh }
\address{Department of Mathematics, Birzeit University, Birzeit,  Palestine}
\email{msaleh@birzeit.edu}

\thanks{$^{(\star)}$ Corresponding author}
\date{}

\begin{abstract} In this paper, we introduce the notion of uniformly $S$-essential ($u$-$S$-essential) submodules. Let $R$ be a commutative ring, $S$ a multiplicative subset of $R$, and $M$ an $R$-module. A submodule $N$ of $M$ is said to be $u$-$S$-essential in $M$ if for any submodule $L$ of $M$, $N\cap L$ is $u$-$S$-torsion implies $L$ is $u$-$S$-torsion. Several properties of this notion are studied. We also introduce the notions of $u$-$S$-uniform modules and $u$-$S$-injective $u$-$S$-envelopes and characterize them in terms of $u$-$S$-essential submodules.

\end{abstract}

\subjclass[2020]{13Cxx, 13C11, 13C12, 16D40.}

\keywords{$u$-$S$-essential submodule, $u$-$S$-uniform module, $u$-$S$-injective module, $u$-$S$-injective $u$-$S$-envelope}

\maketitle

\section{Introduction}
Throughout this paper, $R$ denotes a commutative ring with nonzero identity, all $R$-modules are unitary, and $S$ denotes a multiplicative subset of $R$, that is, $1 \in S$, $0 \notin S$, and $s_1s_2\in S$ for all $s_1,s_2 \in S$. Recall that an $R$-module $M$ is called $S$-torsion if for every $m\in M$, there exists $s\in S$ such that $sm=0$ \cite{WK}. Recently, Zhang \cite{Z} introduced the notion of uniformly $S$-torsion ($u$-$S$-torsion) modules as a refinement of $S$-torsion modules. He defined an $R$-module $M$ to be $u$-$S$-torsion if there exists $s \in S$ such that $sM=0$. He also defined the notions of $u$-$S$-monomorphisms, $u$-$S$-epimorphisms, $u$-$S$-isomorphisms, and $u$-$S$-exact sequences as follows: Let $M, N$, and $L$ be $R$-modules.
    \begin{enumerate}
    \item[(i)] An $R$-homomorphism $f: M \to N$ is called a $u$-$S$-monomorphism ($u$-$S$-epimorphism) if $\text{Ker}(f)$ ($\text{Coker}(f)$) is a $u$-$S$-torsion module.  
    \item[(ii)] An $R$-homomorphism $f: M \to N$ is called a $u$-$S$-isomorphism if $f$ is both a $u$-$S$-monomorphism and a $u$-$S$-epimorphism.
 \item[(iii)] An $R$-sequence $M \xrightarrow{f} N \xrightarrow{g} L$ is said to be $u$-$S$-exact if there exists $s \in S$ such that $s\text{Ker}(g) \subseteq \text{Im}(f)$ and $s\text{Im}(f) \subseteq \text{Ker}(g)$. 
 \end{enumerate} 
After that, Chen et al. \cite{QK} introduced the notion of $u$-$S$-injective modules. They defined an $R$-module $E$ to be $u$-$S$-injective if the induced sequence
 $$0 \to \text{Hom}_{R}(C,E)\to \text{Hom}_R(B,E)\to \text{Hom}_R(A,E) \to 0$$
 is $u$-$S$-exact for any $u$-$S$-exact sequence $0 \to
A\to B\to C \to 0$. Injective modules and $u$-$S$-torsion modules are $u$-$S$-injective \cite[Corollary 4.4]{QK}. 

Essential submodules and injective envelopes play a fundamental role in module theory and homological algebra. They provide important tools for understanding module structure and for constructing minimal injective extensions. Let $M$ be an $R$-module. Recall that a submodule $N$ of $M$ is said to be essential in $M$, denoted by $N\unlhd M$, if for any submodule $L$ of $M$, $N\cap L=0$ implies $L=0$. An injective envelope of $M$ (in the sense of Eckmann-Schopf's) is a monomorphism $f:M\to E$ with $E$ injective and $\text{Im}(f)\unlhd E$. If $f:M\to E$ is an injective envelope of $M$, then $E$ is also called an injective envelope of $M$. The injective envelope of $M$ always exists, and is unique up to isomorphism \cite[Theorem 18.10]{AF}. Many properties of essential submodules and injective envelopes are given in \cite{AF}. Let $M$ be an $R$-module and $\mathcal{A}$ a class of $R$-modules. Recall from \cite[Definition 1.2.1]{Xu} that
\begin{enumerate}
\item[(i)] A linear map $f:M\to A$ with $A\in \mathcal{A}$ is called an $\mathcal{A}$-preenvelope of $M$ if the map
 $$\text{Hom}_R(f,A') : \text{Hom}_R(A,A') \to \text{Hom}_R(M,A')$$ is an epimorphism for any $A'\in \mathcal{A}$.
\item[(ii)] An $\mathcal{A}$-preenvelope $f:M\to A$ is called an $\mathcal{A}$-envelope of $M$ if for each $\alpha \in \text{End}_R(A)$, $f = \alpha f$ implies $\alpha$ is an automorphism.
\end{enumerate}

From \cite[Theorem 1.2.11]{Xu}, if $\mathcal{E}$ denotes the class of all injective $R$-modules, then the concepts of $\mathcal{E}$-envelopes and injective envelopes in the sense of Eckmann-Schopf's coincide.

This paper has two aims. The first aim is to introduce and study the notion of $u$-$S$-essential submodules, and the second 
is to introduce and study the notion of $\mathcal{A}$-$u$-$S$-(pre)envelopes, where $\mathcal{A}$ is a class of $R$-modules, and focus on the case where $\mathcal{A}$ is the class of all $u$-$S$-injective $R$-modules.

This paper is organized as follows: In Section \ref{sec:2}, we first define the notion of $u$-$S$-essential submodules, then we show that if every element of $S$ is a unit in $R$, the notions of essential submodules and $u$-$S$-essential submodules are the same (see Remark \ref{rem1}). We characterize $u$-$S$-torsion modules in terms of $u$-$S$-essential submodules (see Proposition \ref{prop1}(2)). We then introduce the notion of $u$-$S$-uniform modules, and characterize them in terms of $u$-$S$-essential submodules (see Theorem \ref{thm1}). After that, we explore several properties of $u$-$S$-essential submodules. For example, we show in Proposition \ref{prop3} that if $N$ is $u$-$\mathfrak{m}$-essential for every $\mathfrak{m}\in \text{Max}(R)$, then $N$ is essential. We show in Proposition \ref{prop4} that the converse of the last fact is true if $M$ is a prime $R$-module. The condition ''$M$ is a prime $R$-module'' in Proposition \ref{prop4} is necessary as shown in Example \ref{exp5}. We show in Theorem \ref{thm2} and its corollary that $u$-$S$-essentiality is preserved under finite direct sums. Unlike essentiality, $u$-$S$-essentiality is not necessarily preserved under infinite direct sums (see Example \ref{exp6}).

 In Section \ref{sec:3}, we introduce the notion of $\mathcal{A}$-$u$-$S$-(pre)envelopes, where $\mathcal{A}$ is a class of $R$-modules, and focus on the case where $\mathcal{A}$ is the class of all $u$-$S$-injective $R$-modules. We investigate several properties of these notions. For example, we show in Proposition \ref{prop8} that the $\mathcal{A}$-$u$-$S$-envelope, if it exists, is unique up to $u$-$S$-isomorphism. In Theorem \ref{thm4}, we characterize $u$-$S$-injective $u$-$S$-envelopes in terms of $u$-$S$-essential submodules. Following this, we prove in Theorem \ref{thm5} that a finite direct sum of $u$-$S$-injective $u$-$S$-envelopes is a $u$-$S$-injective $u$-$S$-envelope. However, an arbitrary direct sum of $u$-$S$-injective $u$-$S$-envelopes need not be a $u$-$S$-injective $u$-$S$-envelope (see Example \ref{exp7}). The last result of this section (Proposition \ref{prop12}) gives another characterization of $u$-$S$-injective $u$-$S$-envelopes. 
 
 Throughout, $U(R)$ denotes the set of all units of $R$; $\text{reg}(R)$ denotes the set of all regular elements (nonzero divisors) of $R$; $\text{Max}(R)$ denotes the set of all maximal ideals of $R$; $\text{Spec}(R)$ denotes the set of all prime ideals of $R$; $\text{Ann}_{R}(M)$ denotes the annihilator of $M$ in $R$; $E(M)$ denotes the injective envelope of $M$.

\section{ $u$-$S$-essential submodules}\label{sec:2}
 We start this section by introducing the notion of $u$-$S$-essential submodules. 
\begin{definition}\label{def1}
 Let $S$ be a multiplicative subset of a ring $R$ and $M$ an $R$-module. A submodule $N$ of $M$ is called $u$-$S$-essential in $M$, denoted by $N\unlhd^{u-S} M$, if for any submodule $L$ of $M$, $N\cap L$ is $u$-$S$-torsion implies $L$ is $u$-$S$-torsion.
 
\end{definition}

$u$-$S$-essential submodules need not be essential, and essential submodules need not be $u$-$S$-essential, as shown in the following two examples.

\begin{example}\label{exp1}
    Let $R=\mathbb{Z}_6$, $S=\{1,4\}$, and $M=\mathbb{Z}_6$. Then $N:=2\mathbb{Z}_6$ is a $u$-$S$-essential submodule of $M$. To see this, let $L\leq M$. Suppose that $N\cap L$ is $u$-$S$-torsion. The submodules of $M$ are $\{0\},2\mathbb{Z}_6,3\mathbb{Z}_6$, and $\mathbb{Z}_6$. If $L=2\mathbb{Z}_6$ or $L=\mathbb{Z}_6$, then $N\cap L=2\mathbb{Z}_6$ is not $u$-$S$-torsion, a contradiction. So $L=\{0\}$ or $L=3\mathbb{Z}_6$. If $L=\{0\}$ or $L=3\mathbb{Z}_6$, then $N\cap L=\{0\}$ is $u$-$S$-torsion. Thus $N\cap L$ is $u$-$S$-torsion if and only if $L=\{0\}$ or $L=3\mathbb{Z}_6$. But $\{0\}$ and $3\mathbb{Z}_6$ are $u$-$S$-torsion since $4\cdot 0=4\cdot 3=0$. Hence $N$ is $u$-$S$-essential in $M$. However, $N$ is not essential in $M$ since $N\cap 3\mathbb{Z}_6=2\mathbb{Z}_6\cap 3\mathbb{Z}_6=\{0\}$ but $3\mathbb{Z}_6\not=\{0\}$. 
\end{example}

\begin{example}\label{exp2}
    Let $R=\mathbb{Z}$ and $S=\{p^n:n=0,1,2,\cdots\}$, where $p$ is a prime number. Let $M=\frac{\mathbb{Z}_{(p)}}{\mathbb{Z}}$ be the $\mathbb{Z}$-module, where $\mathbb{Z}_{(p)}$ is the localization of $\mathbb{Z}$ at $S$. The submodules of $M$ are of the form $$M_n= \frac{\frac{1}{p^n}\mathbb{Z}}{\mathbb{Z}}=\Big\{\frac{a}{p^n}+\mathbb{Z}\in M\mid a\in \mathbb{Z}\Big\},$$ and they form a chain: $$0=M_0\subset M_1\subset M_2\subset \cdots \subset M.$$  Since for any nonzero submodule $K$ of $M$, $M_1\cap K=M_1\neq 0$, we have $M_1$ is an essential submodule of $M$. However, $M_1$ is not a $u$-$S$-essential submodule of $M$ since $p\in S$ and $p(M_1\cap M)=pM_1=0$ but $M$ is not $u$-$S$-torsion by \cite[Example 2.2(1)]{Z}. 
\end{example}

\begin{lemma}\label{lem1}
Let $R$ be a ring and $S\subseteq U(R)$ a multiplicative set. Then an $R$-module $M$ is $u$-$S$-torsion if and only if $M=0$.    
\end{lemma}

\begin{proof} Let $M$ be a $u$-$S$-torsion module. Then $sM=0$ for some $s\in S$. But $s$ is a unit of $R$, this implies $M=s^{-1}sM=0$. The converse is clear.
\end{proof}

The following remark shows that if $S\subseteq U(R)$, then the notions of $u$-$S$-essential and essential submodules coincide. 

\begin{remark} \label{rem1}
Let $R$ be a ring, $S\subseteq U(R)$ a multiplicative set, $M$ an $R$-module, and $N\leq M$. Then $N$ is $u$-$S$-essential in $M$ if and only if $N$ is essential in $M$. 
\end{remark}

\begin{proof} This follows from Definition \ref{def1} and Lemma \ref{lem1}.
\end{proof}

The following proposition proves that if $N$ is a $u$-$S$-torsion $u$-$S$-essential submodule of $M$, then $M$ is $u$-$S$-torsion, and that $M$ is $u$-$S$-torsion if and only if every submodule of $M$ is $u$-$S$-essential.

\begin{proposition} \label{prop1}
Let $S$ be a multiplicative subset of a ring $R$, $M$ an $R$-module, and $N\leq M$. Then the following statements hold.
\begin{enumerate}
\item[(1)] If $N$ is $u$-$S$-torsion and $u$-$S$-essential submodule of $M$, then $M$ is $u$-$S$-torsion. 
\item[(2)] $M$ is $u$-$S$-torsion if and only if every submodule of $M$ is $u$-$S$-essential.
\end{enumerate} 
\end{proposition}

\begin{proof} 
(1) Since $N\cap M=N$ is $u$-$S$-torsion and $N$ is $u$-$S$-essential in $M$, we have $M$ is $u$-$S$-torsion. \\[0.1cm]
(2) Let $M$ be $u$-$S$-torsion and $K\leq M$. So there is $s\in S$ such that $sM=0$. So for any $L\leq M$, $sL\subseteq sM=0$. Hence, for any $L\leq M$, $L$ is $u$-$S$-torsion. Thus $K$ is $u$-$S$-essential in $M$. Conversely, suppose that every submodule of $M$ is $u$-$S$-essential. Then $\{0\}$ is $u$-$S$-essential in $M$ but $\{0\}$ is $u$-$S$-torsion, so by part (1), $M$ is $u$-$S$-torsion.
\end{proof}

Recall that a nonzero $R$-module $M$ is said to be uniform if the intersection of any two nonzero submodules of $M$ is nonzero \cite{MD}. We now introduce the uniformly $S$-version of uniform modules.

\begin{definition}
  Let $S$ be a multiplicative subset of a ring $R$. An $R$-module $M$ is called $u$-$S$-uniform if $M$ is not $u$-$S$-torsion, and the intersection of any two non-$u$-$S$-torsion submodules of $M$ is non-$u$-$S$-torsion.
    \end{definition}
    
$u$-$S$-uniform modules need not be uniform, and uniform modules need not be $u$-$S$-uniform, as the following two examples show.

\begin{example}\label{exp3}
Let $R=\mathbb{Z}_6$, $S=\{1,4\}$, and $M=\mathbb{Z}_6$. The submodules of $M$ are
        $\{0\},2\mathbb{Z}_6,3\mathbb{Z}_6$, and $\mathbb{Z}_6$. Note that $2\mathbb{Z}_6$ and $\mathbb{Z}_6$ are the only non-$u$-$S$-torsion submodules of $M$. Since $2\mathbb{Z}_6\cap \mathbb{Z}_6=2\mathbb{Z}_6$ is non-$u$-$S$-torsion, $M$ is $u$-$S$-uniform. However, $M$ is not uniform since $2\mathbb{Z}_6$ and $3\mathbb{Z}_6$ are nonzero submodules of $M$ but $2\mathbb{Z}_6\cap 3\mathbb{Z}_6=\{0\}$.
\end{example}

\begin{example}\label{exp4}
Let $R=\mathbb{Z}$, $S=\mathbb{Z}\setminus \{0\}$, and $M=\mathbb{Z}_2$. Then $M$ is uniform since it is a simple $R$-module. However, $M$ is not $u$-$S$-uniform since $M$ is $u$-$S$-torsion.
\end{example}

\begin{remark}\label{rem2}
    Let $R$ be a ring and $S\subseteq U(R)$ a multiplicative set. Then $M$ is $u$-$S$-uniform if and only if $M$ is uniform.
\end{remark}

Recall that a nonzero $R$-module $M$ is uniform if and only if every nonzero submodule of $M$ is essential. The following theorem gives the uniformly $S$-version of this result.

\begin{thm}\label{thm1} 
  Let $S$ be a multiplicative subset of a ring $R$ and $M$ a non-$u$-$S$-torsion $R$-module. Then $M$ is $u$-$S$-uniform if and only if every non-$u$-$S$-torsion submodule of $M$ is $u$-$S$-essential.
\end{thm}

\begin{proof}
   Let $M$ be $u$-$S$-uniform and $N$ be a non-$u$-$S$-torsion submodule of $M$. Let $L\leq M$. Suppose that $L$ is non-$u$-$S$-torsion. Since $M$ is $u$-$S$-uniform, $N\cap L$ is non-$u$-$S$-torsion. So $N$ is $u$-$S$-essential in $M$. 
   
   Conversely, let $N$ and $L$ be non-$u$-$S$-torsion submodules of $M$. Then $N$ is $u$-$S$-essential in $M$. But $L$ is non-$u$-$S$-torsion, so $N\cap L$ is non-$u$-$S$-torsion. Therefore, $M$ is $u$-$S$-uniform.
\end{proof}

Recall that an $R$-module $M$ is said to be prime if $\text{Ann}_{R}(N)=\text{Ann}_{R}(M)$ for every nonzero submodule $N$ of $M$ \cite{YM}. The following result establishes that for any multiplicative subset $S$ of $R$, every essential submodule of a prime $R$-module is $u$-$S$-essential

\begin{proposition}\label{prop2}
Let $S$ be a multiplicative subset of a ring $R$ and $M$ a prime $R$-module. Then every essential submodule of $M$ is $u$-$S$-essential. 
\end{proposition}
\begin{proof}
   Let $M$ be a prime $R$-module and $N$ an essential submodule of $M$. Suppose that $L\leq M$ and $s(N\cap L)=0$ for some $s\in S$. If $L=0$, we are done. If $L\not=0$, then $N\cap L\not=0$ since $N$ is essential in $M$. So $\text{Ann}_{R}(N\cap L)=\text{Ann}_{R}(M)=\text{Ann}_{R}(L)$. Hence $s\in \text{Ann}_{R}(N\cap L)=\text{Ann}_{R}(L)$. Thus $sL=0$. Therefore, $N$ is $u$-$S$-essential in $M$.
\end{proof}

\begin{definition}
Let $R$ be a ring, $\mathfrak{p}$ a prime ideal of $R$, $M$ an $R$-module, and $N\leq M$. We say that $N$ is $u$-$\mathfrak{p}$-essential in $M$ if $N$ is $u$-$(R\setminus \mathfrak{p})$-essential in $M$.
\end{definition}

\begin{proposition}\label{prop3}
Let $R$ be a ring, $M$ an $R$-module, and $N\leq M$. If $N$ is $u$-$\mathfrak{m}$-essential for any $\mathfrak{m}\in \text{Max}(R)$, then $N$ is essential. 
\end{proposition}

\begin{proof}
Let $L\leq M$. Suppose that $N\cap L=0$. Since $N$ is $u$-$\mathfrak{m}$-essential for any $\mathfrak{m}\in \text{Max}(R)$, $L$ is $u$-$(R\setminus \mathfrak{m})$-torsion for any $\mathfrak{m}\in \text{Max}(R)$. By \cite[Lemma 2.18]{MM}, we have $L=0$. Therefore, $N$ is essential.  
\end{proof}

The following proposition provides a local characterization of essential submodules of prime modules. 

\begin{proposition}\label{prop4}
      Let $R$ be a ring, $M$ a prime $R$-module, and $N\leq M$. Then the following statements are equivalent:
      \begin{enumerate}
          \item[(1)] $N$ is essential;
          \item[(2)] $N$ is $u$-$\mathfrak{p}$-essential for any $\mathfrak{p}\in \text{Spec}(R)$;
          \item[(3)] $N$ is $u$-$\mathfrak{m}$-essential for any $\mathfrak{m}\in \text{Max}(R)$.
      \end{enumerate}
\end{proposition}

\begin{proof}
    $(1)\Rightarrow (2)$: This follows from Proposition \ref{prop2}.\\
$(2)\Rightarrow (3)$: Clear.\\
$(3)\Rightarrow (1)$: This follows directly from Proposition \ref{prop3}.
\end{proof}

The condition "$M$ is a prime $R$-module" in Proposition \ref{prop4} is necessary, as the following example shows.

\begin{example}\label{exp5}
Let $R=\mathbb{Z}$ and $\mathfrak{m}=3\mathbb{Z}$. Let $M=\frac{\mathbb{Z}_{(2)}}{\mathbb{Z}}$ be the $\mathbb{Z}$-module given in Example \ref{exp2} with $p=2$, and let $N=\frac{\frac{1}{2}\mathbb{Z}}{\mathbb{Z}}$. Then, as in Example \ref{exp2}, $N$ is an essential submodule of $M$. Since $2\in \text{Ann}_R(N)$ and $2\notin \text{Ann}_R(M)$, $M$ is not a prime $R$-module. Now, if $sM=0$ for some $s\in R\setminus\mathfrak{m}$, then $s\neq 0$ and $s(\frac{1}{2^n}+\mathbb{Z})=0+\mathbb{Z}$ for each $n> 0$. So $\frac{s}{2^n}\in \mathbb{Z}$ for each $n> 0$. This means that $s$ is divisible by $2^n$ for each $n> 0$. So $|s|$ is divisible by $2^n$ for each $n> 0$. Choose a natural number $m$ so that $|s|<2^m$. Then $|s|$ is not divisible by $2^m$, a contradiction. Hence $sM\neq 0$ for all $s\in R\setminus\mathfrak{m}$. That is, $M$ is not $u$-$(R\setminus\mathfrak{m})$-torsion. By the last fact and since $2\in R\setminus \mathfrak{m}$ with $2(N\cap M)=2N=0$, we have $N$ is not $u$-$\mathfrak{m}$-essential in $M$.
\end{example}

The following lemma is used frequently in this paper.

\begin{lemma}\label{lem2}
Let $S$ be a multiplicative subset of a ring $R$ and $M$ an $R$-module. Let $K,N\leq M$ be such that $K\subseteq N$. If $N$ is $u$-$S$-torsion, then $K$ is $u$-$S$-torsion.    
\end{lemma}

\begin{proof}
    Suppose that $N$ is $u$-$S$-torsion. Then $sN=0$ for some $s\in S$. So $sK\subseteq sN=0$. Hence $sK=0$. Thus $K$ is $u$-$S$-torsion.
\end{proof}

\begin{proposition}\label{prop5}
Let $S$ be a multiplicative subset of a ring $R$ and $M$ an $R$-module. Let $K\leq N\leq M$ and $H\leq M$. Then the following statements hold.
\begin{enumerate}
\item[(1)] $K\unlhd^{u-S} M$ if and only if $K\unlhd^{u-S} N$ and $N\unlhd^{u-S} M$. 
\item[(2)] $H\cap K\unlhd^{u-S} M$ if and only if $H\unlhd^{u-S} M$ and $K\unlhd^{u-S} M$. 
\end{enumerate}
\end{proposition}
\begin{proof} (1) $(\Rightarrow)$. First, we show that $K\unlhd^{u-S} N$. Let $L\leq N$. Suppose that $K\cap L$ is $u$-$S$-torsion. Since $L\leq M$ and $K\unlhd^{u-S} M$, we have $L$ is $u$-$S$-torsion. So $K\unlhd^{u-S} N$. Next, we show that $N\unlhd^{u-S} M$. Let $L\leq M$. Suppose that $N\cap L$ is $u$-$S$-torsion. Since $K\cap L\subseteq N\cap L$, $K\cap L$ is $u$-$S$-torsion by Lemma \ref{lem2}. But $K\unlhd^{u-S} M$, so $L$ is $u$-$S$-torsion. Thus $N\unlhd^{u-S} M$. \\
$(\Leftarrow)$. Let $L\leq M$. Suppose that $K\cap L$ is $u$-$S$-torsion. Then $K\cap N\cap L=K\cap L$ is $u$-$S$-torsion. Since $N\cap L\leq N$ and $K\unlhd^{u-S} N$, we have $N\cap L$ is $u$-$S$-torsion. But $N\unlhd^{u-S} M$, so $L$ is $u$-$S$-torsion. Hence $K\unlhd^{u-S} M$.\\[0.1cm]
(2) $(\Rightarrow)$. Since $H\cap K\leq H\leq M$, $H\cap K\leq K\leq M$, and $H\cap K\unlhd^{u-S} M$, then by part (1), $H\unlhd^{u-S} M$ and $K\unlhd^{u-S} M$. \\
 $(\Leftarrow)$. Let $L\leq M$. Suppose that $H\cap K\cap L$ is $u$-$S$-torsion. Since $H\unlhd^{u-S}M$, $K\cap L$ is $u$-$S$-torsion. But $K\unlhd^{u-S}M$, hence $L$ is $u$-$S$-torsion. Therefore, $H\cap K\unlhd^{u-S} M$.         
\end{proof}

\begin{lemma}\label{lem3}
Let $S$ be a multiplicative subset of a ring $R$, $f:M\to N$ an $R$-homomorphism, and $H\leq M$. Then the following statements hold.
 \begin{enumerate}
 \item[(1)] If $H$ is $u$-$S$-torsion, then $f(H)$ is $u$-$S$-torsion.
\item[(2)] If $f$ is a $u$-$S$-monomorphism, then the converse of (1) holds. 
    \end{enumerate}
\end{lemma}

\begin{proof}
 (1) Suppose that $H$ is $u$-$S$-torsion. Then $sH=0$ for some $s\in S$. So $sf(H)=f(sH)=0$. Hence $f(H)$ is $u$-$S$-torsion.\\[0.1cm]
 (2) Suppose that $f(H)$ is $u$-$S$-torsion. Then $sf(H)=0$ for some $s\in S$. So $sH\subseteq \text{Ker}(f)$. But $\text{Ker}(f)$ is $u$-$S$-torsion since  $f$ is a $u$-$S$-monomorphism. Hence $sH$ is $u$-$S$-torsion by Lemma \ref{lem2}. Thus $s'sH=0$ for some $s'\in S$. Therefore, $H$ is $u$-$S$-torsion. 
\end{proof}

\begin{proposition}\label{prop6}
    Let $S$ be a multiplicative subset of a ring $R$ and $f:M\to N$ an $R$-homomorphism. Then the following statements hold. 
    \begin{enumerate}
 \item[(1)] If $K\unlhd^{u-S} N$, then $f^{-1}(K)\unlhd^{u-S} M$.
\item[(2)] If $H\unlhd^{u-S} M$ and $f$ is a $u$-$S$-monomorphism, then $f(H)\unlhd^{u-S} f(M)$.
    \end{enumerate}
   
\end{proposition}
\begin{proof} (1) Let $L\leq M$. Suppose that $L\cap f^{-1}(K)$ is $u$-$S$-torsion. It is easy to chick that $f\big(L\cap f^{-1}(K)\big)=f(L)\cap K$. By Lemma \ref{lem3}(1), we have $f(L)\cap K$ is $u$-$S$-torsion. Since $K\unlhd^{u-S} N$, $f(L)$ is $u$-$S$-torsion. So $sf(L)=0$ for some $s\in S$. This implies that $sL\subseteq f^{-1}(0)\subseteq f^{-1}(K)$. Hence $sL=sL\cap f^{-1}(K)\subseteq L\cap f^{-1}(K)$. But $L\cap f^{-1}(K)$ is $u$-$S$-torsion, so $sL$ is $u$-$S$-torsion by Lemma \ref{lem2}. Thus $L$ is $u$-$S$-torsion. Therefore, $f^{-1}(K)\unlhd^{u-S} M$.\\[0.1cm]
(2) Let $L\leq f(M)$. Assume that $f(H)\cap L$ is $u$-$S$-torsion. Since $f$ is a $u$-$S$-monomorphism and $f\big(H\cap f^{-1}(L)\big)=f(H)\cap L$ is $u$-$S$-torsion, then $H\cap f^{-1}(L)$ is $u$-$S$-torsion by Lemma \ref{lem3}(2). But $H\unlhd^{u-S} M$, so $f^{-1}(L)$ is $u$-$S$-torsion. Since $L\leq f(M)$, $L=f(f^{-1}(L))$. Hence $L$ is $u$-$S$-torsion by Lemma \ref{lem3}(1). Thus $f(H)\unlhd^{u-S} f(M)$.
\end{proof}

The following theorem and its corollary show that $u$-$S$-essentiality is preserved under finite direct sums.

\begin{thm} \label{thm2}
 Let $S$ be a multiplicative subset of a ring $R$. Let $M_1$ and $M_2$ be $R$-modules and let $K_{1}\leq M_{1}$ and $K_{2}\leq M_{2}$. Then
  $K_{1}\oplus K_{2}\unlhd^{u-S} M_{1}\oplus M_{2}$ if and only if $K_{1}\unlhd^{u-S} M_{1}$ and $K_{2}\unlhd^{u-S} M_{2}$.

\end{thm}

\begin{proof}
$(\Rightarrow)$. Suppose that $K_{1}\oplus K_{2}\unlhd^{u-S} M_{1}\oplus M_{2}$. If $K_{1}$ is not $u$-$S$-essential in $M_{1}$, then there is $L_{1}\leq M_{1}$ such that $K_{1}\cap L_{1}$ is $u$-$S$-torsion but $L_{1}$ is not $u$-$S$-torsion. Now $$(K_{1}\oplus K_{2})\cap (L_{1}\oplus 0)=(K_{1}\cap L_{1})\oplus (K_{2}\cap 0)=(K_{1}\cap L_{1})\oplus 0.$$ So $(K_{1}\oplus K_{2})\cap (L_{1}\oplus 0)$ is $u$-$S$-torsion. But $K_{1}\oplus K_{2}\unlhd^{u-S} M_{1}\oplus M_{2}$, so $L_{1}\oplus 0$ is $u$-$S$-torsion, which implies $L_{1}$ is $u$-$S$-torsion, a contradiction. Thus $K_{1}\unlhd^{u-S} M_{1}$. Similarly, we can show that $K_{2}\unlhd^{u-S} M_{2}$.\\
  $(\Leftarrow)$. Let $\pi_{1}:M_{1}\oplus M_{2}\to M_{1}$ and $\pi_{2}:M_{1}\oplus M_{2}\to M_{2}$ be the natural projections. Since $K_{1}\unlhd^{u-S} M_{1}$ and $K_{2}\unlhd^{u-S} M_{2}$, then by Proposition \ref{prop6}(1), we have $\pi_{1}^{-1}(K_{1})\unlhd^{u-S} M_{1}\oplus M_{2}$ and $\pi_{2}^{-1}(K_{2})\unlhd^{u-S} M_{1}\oplus M_{2}$. But $\pi_{1}^{-1}(K_{1})=K_1\oplus M_2$ and $\pi_{2}^{-1}(K_{2})=M_1\oplus K_2$. So $K_1\oplus M_2\unlhd^{u-S} M_{1}\oplus M_{2}$ and $M_1\oplus K_2\unlhd^{u-S} M_{1}\oplus M_{2}$. By Proposition \ref{prop5}(2), we have $$(K_1\oplus M_2)\cap (M_1\oplus K_2)\unlhd^{u-S} M_{1}\oplus M_{2}.$$ But $K_{1}\oplus K_{2}=(K_1\oplus M_2)\cap (M_1\oplus K_2)$. Thus $K_{1}\oplus K_{2}\unlhd^{u-S} M_{1}\oplus M_{2}$.
\end{proof}

\begin{corollary}\label{cor1}
    Let $S$ be a multiplicative subset of a ring $R$. For each $i=1,2,\cdots, n$, let $M_i$ be an $R$-module and $K_i\leq M_i$. Then $\bigoplus\limits_{i=1}^{n}K_i \unlhd^{u-S} \bigoplus\limits_{i=1}^{n}M_i$ if and only if $K_{i}\unlhd^{u-S} M_{i}$ for each $i=1,2,\cdots, n$.
\end{corollary}

Corollary \ref{cor1} may fail for infinite direct sums, as illustrated by the following example.

\begin{example}\label{exp6} 
 Let $R= \mathbb{Z}$, $p$ a prime in $\mathbb{Z}$, and $S = \{p^n \mid n=0,1,2,\cdots\}$. For each $n \geq 1$, let $M_n = \frac{\mathbb{Z}}{p^n\mathbb{Z}}$ and $K_n=\{0+p^n\mathbb{Z}\}$. Then each $M_n$ is $u$-$S$-torsion. So by Proposition \ref{prop1}(2), $K_n$ is $u$-$S$-essential in $M_n$ for each $n\geq 1$. But $\bigoplus\limits_{n=1}^{\infty}K_n=0$ is not $u$-$S$-essential in $\bigoplus\limits_{n=1}^{\infty}M_n$ since $0\cap \bigoplus\limits_{n=1}^{\infty}M_n=0$ is $u$-$S$-torsion but $\bigoplus\limits_{n=1}^{\infty}M_n$ is not $u$-$S$-torsion. To see this, let $M=\bigoplus\limits_{n=1}^{\infty}M_n$ and suppose there is $p^m\in S$ such that $p^mM=0$. Let $x$ be the element in $M$ such that the $(m+1)$-\text{th} component is $1+p^{m+1}\mathbb{Z}$ and all other components are $0$. Then $p^mx=0$ but $p^m(1+p^{m+1}\mathbb{Z})=p^m+p^{m+1}\mathbb{Z}\neq 0+p^{m+1}\mathbb{Z}$, so $p^mx\neq 0$, a contradiction. Hence $M$ is not $u$-$S$-torsion.
\end{example}

Let $S$ be a multiplicative subset of a ring $R$ and $M$ an $R$-module. The set $$\text{tor}_{S}(M)=\{m\in M\mid sm=0 ~\text{for some}~ s\in S\}$$ is a submodule of $M$, called the $S$-torsion submodule of $M$ \cite{WK}. The following theorem gives a necessary and sufficient condition for a submodule $N$ of $M$ to be $u$-$S$-essential under the condition that $\text{tor}_{S}(M)$ is $u$-$S$-torsion.

\begin{thm} \label{thm3}
Let $S$ be a multiplicative subset of a ring $R$, $M$ an $R$-module, and $N\leq M$. Suppose that $\text{tor}_{S}(M)$ is $u$-$S$-torsion. Then $N\unlhd^{u-S}M$ if and only if for each $x\in M\setminus \text{tor}_{S}(M)$, $N\cap Rx$ is not $u$-$S$-torsion.
\end{thm}
\begin{proof}
   $(\Rightarrow)$. Let $x\in M\setminus \text{tor}_{S}(M)$. If $N\cap Rx$ is $u$-$S$-torsion, then $Rx$ is $u$-$S$-torsion since $N\unlhd^{u-S}M$. So there is $s\in S$ such that $sRx=0$ but then $sx=0$ and so $x\in \text{tor}_{S}(M)$, a contradiction. Hence $N\cap Rx$ is not $u$-$S$-torsion.\\
   $(\Leftarrow)$. Let $L\leq M$. Suppose that $L$ is not $u$-$S$-torsion. Then $L\nsubseteq \text{tor}_{S}(M)$ by Lemma \ref{lem2} and since $\text{tor}_{S}(M)$ is $u$-$S$-torsion. Take $x\in L\setminus \text{tor}_{S}(M)$. By hypothesis, $N\cap Rx$ is not $u$-$S$-torsion. But $N\cap Rx\subseteq N\cap L$, so $N\cap L$ is not $u$-$S$-torsion by Lemma \ref{lem2}. Therefore, $N\unlhd^{u-S}M$.
\end{proof}

The following result is a special case of Theorem \ref{thm3}. 

\begin{corollary}\cite[Lemma 5.19]{AF}.
Let $R$ be a ring, $M$ an $R$-module, and $N\leq M$. Then $N\unlhd M$ if and only if for each $0\not= x\in M$, there is $r\in R$ such that $0\not=rx\in N$.
\end{corollary}
\begin{proof}
Take $S=\{1\}$. Then $\text{tor}_{S}(M)=\{0\}$ is $u$-$S$-torsion and the result follows from Lemma \ref{lem1}, Remark \ref{rem1}, and Theorem \ref{thm3}.
\end{proof}

Let $R$ be a commutative ring and $M$ an $R$-module. Recall that the idealization of $M$ is the commutative ring
$R\ltimes M=R\times M$ with component-wise addition and multiplication defined by $(a,x)(b,y)=(ab,ay+bx)$ \cite{AW}.
The canonical embedding $i_{R}:R\hookrightarrow R\ltimes M$ defined by $r\mapsto (r,0)$, $r\in R$, induces an $R$-module structure on $R\ltimes M$ via the action $$r\cdot (a,x)=(r,0)(a,x)=(ra,rx),~~~r\in R,~ (a,x)\in R\ltimes M.$$ 
 
The following example shows that the condition ''$\text{tor}_{S}(M)$ is $u$-$S$-torsion'' in Theorem \ref{thm3} is necessary.

\begin{example}
 Let $R=\mathbb{Z}$, $S=\{1,2,3,\cdots\}$, and $M=\mathbb{Z}\ltimes\frac{\mathbb{Q}}{\mathbb{Z}}$. Then $\text{tor}_{S}(M)=0\ltimes \frac{\mathbb{Q}}{\mathbb{Z}}$ is not $u$-$S$-torsion. Let $N=R\big(1,\frac{1}{2}+\mathbb{Z}\big)$. Then $N$ is not $u$-$S$-essential in $M$ since $N\cap \big(0\ltimes \frac{\mathbb{Q}}{\mathbb{Z}}\big)=0$ but $0\ltimes \frac{\mathbb{Q}}{\mathbb{Z}}$ is not $u$-$S$-torsion. However, if $x=\big(k,\frac{m}{n}+\mathbb{Z}\big)\in M\setminus \text{tor}_{S}(M)$, then $k\not=0$, and if we take $r=2n\in R$, then $$rx=(2nk,2m+\mathbb{Z})=(2nk,0+\mathbb{Z})=\bigg(2nk,\frac{2nk}{2}+\mathbb{Z}\bigg)=2nk\bigg(1,\frac{1}{2}+\mathbb{Z}\bigg)\in N\cap Rx,$$ and for each $s\in S$, we have $$(0,0+\mathbb{Z})\neq (2snk,2sm+\mathbb{Z})=srx\in s(N\cap Rx).$$ 
 Thus $N\cap Rx$ is not $u$-$S$-torsion.
\end{example}

Next, we define the notion of $u$-$S$-essential $u$-$S$-monomorphisms. 

\begin{definition}
    Let $S$ be a multiplicative subset of a ring $R$. A $u$-$S$-monomorphism $f: M \to N$ is said to be $u$-$S$-essential if $\text{Im}(f)\unlhd^{u-S} N$.
\end{definition}

The final result of this section characterizes $u$-$S$-essential submodules.

\begin{proposition}\label{prop7}
   Let $S$ be a multiplicative subset of a ring $R$, $M$ an $R$-module, and $K\leq M$. Then the following statements are equivalent:
     \begin{enumerate}
    \item[(1)] $K\unlhd^{u-S} M$;
    \item[(2)] The inclusion map $i_{K}: K \to M$ is a $u$-$S$-essential monomorphism;
\item[(3)] For any $R$-module $N$ and for any $R$-homomorphism $h:M\to N$, if $hi_{K}$ is a $u$-$S$-monomorphism, then $h$ is a $u$-$S$-monomorphism.
\end{enumerate}
\end{proposition}

\begin{proof}
    $(1)\Leftrightarrow (2)$: Clear.\\ 
    $(1)\Rightarrow (3)$: Let $K\unlhd^{u-S} M$ and $h:M\to N$ be an $R$-homomorphism. Suppose that $hi_{K}$ is a $u$-$S$-monomorphism. Then $\text{Ker} (hi_{K})$ is $u$-$S$-torsion. But $\text{Ker} (hi_{K})=K\cap \text{Ker}(h)$, so $K\cap \text{Ker}(h)$ is $u$-$S$-torsion. Since $K\unlhd^{u-S} M$, $\text{Ker} (h)$ is $u$-$S$-torsion. Thus $h$ is a $u$-$S$-monomorphism.\\
    $(3)\Rightarrow (1)$: Let $L\leq M$. Suppose that $K\cap L$ is $u$-$S$-torsion. Since $L=\text{Ker} (\eta_{L})$, where $\eta_{L}: M \to \frac{M}{L}$ is the natural map and $\text{Ker} (\eta_{L}i_{K})=K\cap \text{Ker}(\eta_{L})=K\cap L$, we have $\text{Ker}(\eta_{L}i_{K})$ is $u$-$S$-torsion. That is, $\eta_{L}i_{K}$ is a $u$-$S$-monomorphism. So by (3), we have $\eta_{L}$ is a $u$-$S$-monomorphism. Hence $L=\text{Ker} (\eta_{L})$ is $u$-$S$-torsion. Therefore, $K\unlhd^{u-S} M$.
\end{proof}

\section{Uniformly $S$-Injective Uniformly $S$-Envelopes}\label{sec:3}

In this section, we introduce the notion of $\mathcal{A}$-$u$-$S$-(pre)envelopes, where $\mathcal{A}$ is a class of $R$-modules, and focus on the case where $\mathcal{A}$ is the class of all $u$-$S$-injective $R$-modules. We start this section by recalling the following definition:

\begin{definition}\label{def2}\cite[Definition 1.2.1]{Xu}. Let $R$ be a ring, $M$ an $R$-module, and $\mathcal{A}$ a class of $R$-modules. 
 \begin{enumerate}
\item[(i)] A linear map $f:M\to A$ with $A\in \mathcal{A}$ is called an $\mathcal{A}$-preenvelope of $M$ if the map
 $$\text{Hom}_R(f,A') : \text{Hom}_R(A,A') \to \text{Hom}_R(M,A')$$ is an epimorphism for any $A'\in \mathcal{A}$.
\item[(ii)] An $\mathcal{A}$-preenvelope $f:M\to A$ is called an $\mathcal{A}$-envelope of $M$ if for each $\alpha \in \text{End}_R(A)$, $f = \alpha f$ implies $\alpha$ is an automorphism.
 \end{enumerate}
 \end{definition}
Now, we give the uniformly $S$-version of $\mathcal{A}$-(pre)envelopes, where $\mathcal{A}$ is a class of $R$-modules.
 
 \begin{definition}\label{def3}
 Let $S$ be a multiplicative subset of a ring $R$, $M$ an $R$-module, and $\mathcal{A}$ a class of $R$-modules.
 \begin{enumerate}
     \item[(1)] A linear map $f:M\to A$ with $A\in \mathcal{A}$ is called an $\mathcal{A}$-$u$-$S$-preenvelope of $M$ if the map
 $$\text{Hom}_R(f,A') : \text{Hom}_R(A,A') \to \text{Hom}_R(M,A')$$ is a $u$-$S$-epimorphism for any $A'\in \mathcal{A}$.
     \item[(2)] An $\mathcal{A}$-$u$-$S$-preenvelope $f:M\to A$ is called an $\mathcal{A}$-$u$-$S$-envelope of $M$ if for each $\alpha \in \text{End}_R(A)$, $sf=\alpha f$ for some $s\in S$ implies $\alpha$ is a $u$-$S$-isomorphism. 
     \item[(3)] If $A\in \mathcal{A}$ and $f:M\to A$ is an $\mathcal{A}$-$u$-$S$-(pre)envelope of $M$, then we also say $A$ is an $\mathcal{A}$-$u$-$S$-(pre)envelope of $M$.
     \item[(4)] If $u$-$S$-$\mathcal{I}$ is the class of all $u$-$S$-injective $R$-modules, then a $u$-$S$-$\mathcal{I}$-$u$-$S$-(pre)envelope $f:M\to A$ is called a $u$-$S$-injective $u$-$S$-(pre)envelope.  
 \end{enumerate}
 \end{definition}

\begin{remark}\label{rem3}
   Let $R$ be a ring, $M$ an $R$-module, $\mathcal{A}$ a class of $R$-modules, and $A\in \mathcal{A}$. If $S=\{1\}$, then a linear map $f: M\to A$ is an $\mathcal{A}$-$u$-$S$-(pre)envelope if and only if $f$ is an $\mathcal{A}$-(pre)envelope.
\end{remark}

The following proposition shows that the $\mathcal{A}$-$u$-$S$-envelope of $M$, if it exists, is unique up to $u$-$S$-isomorphism.

\begin{proposition}\label{prop8}
 Let $S$ be a multiplicative subset of a ring $R$ and $M$ an $R$-module. If $f_1:M\to A_1$ and $f_2:M\to A_2$ are $\mathcal{A}$-$u$-$S$-envelopes of $M$, then $A_1$ is $u$-$S$-isomorphic to $A_2$.
\end{proposition}

\begin{proof}
 Since $f_1:M\to A_1$ and $f_2:M\to A_2$ are $\mathcal{A}$-$u$-$S$-preenvelopes of $M$, the maps
 $$f_1^* : \text{Hom}_R(A_1,A_2) \to \text{Hom}_R(M,A_2)~\text{and }~f_2^* : \text{Hom}_R(A_2,A_1) \to \text{Hom}_R(M,A_1)$$ are $u$-$S$-epimorphisms. So $$s_1\text{Hom}_R(M,A_2)\subseteq \text{Im}(f_1^*) ~\text{and }~ s_2\text{Hom}_R(M,A_1)\subseteq \text{Im}(f_2^*)$$ for some $s_1,s_2\in S$. Hence $s_1f_2=g_1f_1$ and $s_2f_1=g_2f_2$ for some $R$-homomorphisms $g_1:A_1\to A_2$ and $g_2:A_2\to A_1$. Let $s=s_1s_2$. Then $$sf_1=s_1s_2f_1=s_1g_2f_2=g_2s_1f_2=g_2g_1f_1.$$ Similarly, we have $sf_2=g_1g_2f_2$. Since $f_1:M\to A_1$ and $f_2:M\to A_2$ are $\mathcal{A}$-$u$-$S$-envelopes of $M$, then $g_2g_1:A_1\to A_1$ and $g_1g_2:A_2\to A_2$ are $u$-$S$-isomorphisms. We claim that $g_1:A_1\to A_2$ is a $u$-$S$-isomorphism. Since $\text{Ker}(g_1)\subseteq \text{Ker}(g_2g_1)$ and $\text{Ker}(g_2g_1)$ is $u$-$S$-torsion, then by Lemma \ref{lem2}, $\text{Ker}(g_1)$ is $u$-$S$-torsion. That is, $g_1$ is a $u$-$S$-monomorphism. Next, since $g_1g_2$ is a $u$-$S$-epimorphism, there is $t\in S$ such that $$tA_2\subseteq \text{Im}(g_1g_2)=g_1(\text{Im}(g_2))\subseteq \text{Im}(g_1).$$ Hence $g_1$ is a $u$-$S$-epimorphism. Thus $g_1$ is a $u$-$S$-isomorphism. Therefore, $A_1$ is $u$-$S$-isomorphic to $A_2$.
\end{proof}

Recall that a short $u$-$S$-exact sequence $0\to M \xrightarrow{f} N \xrightarrow{g} L\to 0$ is said to be $u$-$S$-split (with respect to $s$) if there exist $s\in S$ and an $R$-homomorphism $f':N \to M$ such that $f'f=s1_{M}$, where $1_{M}:M\to M$ is the identity map on $M$ \cite{ZQ}. Equivalently, if there exist $s\in S$ and an $R$-homomorphism $g':L\to N$ such that $gg'=s1_{L}$ \cite[Lemma 2.4]{ZQ}. The following proposition proves that the $\mathcal{A}$-$u$-$S$-envelope of $M$, if it exists, is a $u$-$S$-direct summand of any $\mathcal{A}$-$u$-$S$-preenvelope of $M$.

\begin{proposition}\label{prop9} 
     Let $S$ be a multiplicative subset of a ring $R$ and $M$ an $R$-module. If $f:M\to A$ is an $\mathcal{A}$-$u$-$S$-envelope of $M$ and $g:M\to A'$ is an $\mathcal{A}$-$u$-$S$-preenvelope of $M$, then $A$ is a $u$-$S$-direct summand of $A'$.
     \end{proposition}
\begin{proof}
   Let $f:M\to A$ be an $\mathcal{A}$-$u$-$S$-envelope of $M$ and $g:M\to A'$ be an $\mathcal{A}$-$u$-$S$-preenvelope of $M$. Then the maps
        \begin{center}
        $f^* : \text{Hom}_R(A,A') \to \text{Hom}_R(M,A')$ ~ and ~ $g^* : \text{Hom}_R(A',A) \to \text{Hom}_R(M,A)$
        \end{center} are $u$-$S$-epimorphisms. So there exist $s_1,s_2\in S$ such that $s_1g=h_1f$ and $s_2f=h_2g$ for some $R$-homomorphisms $h_1:A\to A'$ and $h_2:A'\to A$. We have the following diagram:  \[\xymatrix{
                              &A \ar[d]^{h_1}\\
 M \ar@{>}[ru]^-{ f}\ar[r]^{g}\ar@{>}[dr]_{ f}&A' \ar[d]^{h_2}\\
           &A
}\] Let $s=s_1s_2$. Then $$sf=s_1s_2f=s_1h_2g=h_2s_1g=h_2h_1f$$ Since $f:M\to A$ is an $\mathcal{A}$-$u$-$S$-envelope of $M$, $h:=h_2h_1$ is a $u$-$S$-isomorphism. By \cite[Lemma 2.1]{ZQ}, there is a $u$-$S$-isomorphism $h':A\to A$ such that $hh'=h'h=t1_{A}$ for some $t\in S$. Since $(h'h_2)h_1=t1_A$, $tA\subseteq \text{Im}((h'h_2)h_1)\subseteq \text{Im}(h'h_2)$. So $h'h_2:A'\to A$ is a $u$-$S$-epimorphism. Let $B=\text{Ker}(h'h_2)$, then the $u$-$S$-exact sequence $0\to B\to A'\xrightarrow{h'h_2}A\to 0$ is $u$-$S$-split. Thus by \cite[Lemma 2.8]{KMOZ}, $A'$ is $u$-$S$-isomorphic to $B\oplus A$. This means that $A$ is a $u$-$S$-direct summand of $A'$.
\end{proof}

Let $S$ be a multiplicative subset of a ring $R$. Recall that an $R$-module $E$ is $u$-$S$-injective if and only if for any $u$-$S$-monomorphism $f:A\to B$, there exists $s\in S$ such that for any
 $R$-homomorphism $h :A\to E$, there exists an $R$-homomorphism $g:B\to E$
 such that $sh=gf$ \cite[Proposition 2.5]{ZQ}. 
 
The following result characterizes $u$-$S$-injective $u$-$S$-preenvelopes and shows that every $R$-module has a $u$-$S$-injective $u$-$S$-preenvelope.
 
\begin{proposition}\label{prop10}
     Let $S$ be a multiplicative subset of a ring $R$ and $M$ an $R$-module. Then the following statements hold.
  \begin{enumerate}
     \item[(1)] An $R$-homomorphism $f:M\to E$ is a $u$-$S$-injective $u$-$S$-preenvelope of $M$ if and
 only if $f$ is a $u$-$S$-monomorphism and $E$ is $u$-$S$-injective.
     \item[(2)] Every $R$-module has a $u$-$S$-injective $u$-$S$-preenvelope.
 \end{enumerate}
\end{proposition}

\begin{proof} (1) Suppose that $f:M\to E$ is a $u$-$S$-injective $u$-$S$-preenvelope. Then $E$ is $u$-$S$-injective. Let $g:M\to E'$ be a monomorphism with $E'$ injective. Since the map $f^* : \text{Hom}_R(E,E') \to \text{Hom}_R(M,E')$ is a $u$-$S$-epimorphism, $s\text{Hom}_R(M,E')\subseteq \text{Im}(f^*)$ for some $s\in S$. So $sg=hf$ for some $R$-homomorphism $h:E\to E'$. Let $x\in \text{Ker}(f)$. Then $f(x)=0$ and so $g(sx)=sg(x)=hf(x)=0$. Since $g$ is a monomorphism, we have $sx=0$. Hence $s \text{Ker}(f)=0$. That is, $f$ is a $u$-$S$-monomorphism. 

Conversely, suppose that $f:M\to E$ is a $u$-$S$-monomorphism and $E$ is $u$-$S$-injective. Let $E'$ be any $u$-$S$-injective module. Then there exists $s'\in S$ such that for any
 $R$-homomorphism $h :M\to E'$, there exists an $R$-homomorphism $g:E\to E'$ such that $s'h=gf$. This means that the map $$f^* : \text{Hom}_R(E,E') \to \text{Hom}_R(M,E')$$ is a $u$-$S$-epimorphism. Thus $f:M\to E$ is a $u$-$S$-injective $u$-$S$-preenvelope of $M$.\\[0.1cm]
 (2) Let $M$ be any $R$-module. Then there is a monomorphism $i:M\to E$ with $E$ injective. Since every monomorphism is a $u$-$S$-monomorphism and every injective is $u$-$S$-injective, then $i:M\to E$ is a $u$-$S$-monomorphism with $E$ $u$-$S$-injective. Therefore, by part (1), $i:M\to E$ is a $u$-$S$-injective $u$-$S$-preenvelope of $M$.  
\end{proof} 

The following main result characterizes $u$-$S$-injective $u$-$S$-envelopes in terms of $u$-$S$-essential submodules.

\begin{thm}\label{thm4}
    Let $S$ be a multiplicative subset of a ring $R$ and $M$ an $R$-module. Then the following statements are equivalent:
    \begin{enumerate}
        \item[(1)] $f:M\to E$ is a $u$-$S$-injective $u$-$S$-envelope;
        \item[(2)] $f:M\to E$ is a $u$-$S$-monomorphism with $E$ $u$-$S$-injective and $\text{Im}(f)\unlhd^{u-S} E$.
    \end{enumerate} 
\end{thm}

\begin{proof}
  $(1)\Rightarrow (2)$: Let $f: M\to E$ be a $u$-$S$-injective $u$-$S$-envelope. Then by Proposition \ref{prop10}(1), $f:M\to E$ is a $u$-$S$-monomorphism and $E$ is $u$-$S$-injective. Now, we show that $\text{Im}(f)\unlhd^{u-S} E$. Let $L\leq E$. Suppose that $L\cap \text{Im}(f)$ is $u$-$S$-torsion. Consider the sequence $M\xrightarrow{f} E\xrightarrow{\eta_L} \frac{E}{L}$. Take $x\in \text{Ker}(\eta_Lf)$. Then $f(x)+L=\eta_Lf(x)=0+L$. So $f(x)\in L$. Consequently, $f(\text{Ker}(\eta_Lf))\subseteq L\cap \text{Im}(f)$. But $L\cap \text{Im}(f)$ is $u$-$S$-torsion implies $f(\text{Ker}(\eta_Lf))$ is $u$-$S$-torsion by Lemma \ref{lem2}. Since $f$ is a $u$-$S$-monomorphism, $\text{Ker}(\eta_Lf)$ is $u$-$S$-torsion by Lemma \ref{lem3}(2). Hence $\eta_Lf:M\to \frac{E}{L}$ is a $u$-$S$-monomorphism. Since $E$ is $u$-$S$-injective, there is an $R$-homomorphism $g:\frac{E}{L}\to E$ such that the following diagram 
   \[\xymatrix{
& E &\\
 &M\ar[u]^{sf} \ar[r]_{\eta_Lf}&\frac{E}{L}\ar@{->}[lu]_{g}
}\] 
commutes for some $s\in S$. So $sf=g\eta_Lf$. Since $f$ is a $u$-$S$-injective $u$-$S$-envelope, $g\eta_L$ is a $u$-$S$-isomorphism. So $\text{Ker}(g\eta_L)$ is $u$-$S$-torsion. But $L=\text{Ker}(\eta_L)\subseteq \text{Ker}(g\eta_L)$. Hence, again by Lemma \ref{lem2}, $L$ is $u$-$S$-torsion. Thus $\text{Im}(f)\unlhd^{u-S} E$. \\[0.1cm]
$(2)\Rightarrow (1)$: Let $f: M\to E$ be a $u$-$S$-monomorphism with $E$ $u$-$S$-injective and $\text{Im}(f)\unlhd^{u-S} E$. By Proposition \ref{prop10}(1), $f$ is a $u$-$S$-injective $u$-$S$-preenvelope. 

Next, let $\alpha\in \text{End}_R(E)$. Suppose that $sf=\alpha f$ for some $s\in S$. Let $y=f(x)\in \text{Ker}(\alpha)\cap \text{Im}(f)$. Then $sy=sf(x)=\alpha f(x)=0$. So $s\big(\text{Ker}(\alpha)\cap \text{Im}(f)\big)=0$. This means that $\text{Ker}(\alpha)\cap \text{Im}(f)$ is $u$-$S$-torsion. But $\text{Im}(f)\unlhd^{u-S} E$, so $\text{Ker}(\alpha)$ is $u$-$S$-torsion. Hence $\alpha$ is a $u$-$S$-monomorphism. Since $E$ is $u$-$S$-injective, then by  \cite[Corollary 2.7(1)]{ZQ}, the $u$-$S$-exact sequence 
\begin{center}
    $0\to E\xrightarrow{\alpha}E\to \text{Coker}(\alpha)\to 0$
\end{center} is $u$-$S$-split. So there is an $R$-homomorphism $\beta:E\to E$ such that $\beta \alpha=t1_E$ for some $t\in S$. Since $s\beta f=\beta sf=\beta \alpha f=tf$, then $t\big(\text{Ker}(\beta)\cap \text{Im}(f)\big)=0$. Again, since $\text{Im}(f)\unlhd^{u-S} E$, we have $t'\text{Ker}(\beta)=0$ for some $t'\in S$. Let $e\in E$. Then $t\beta(e)=\beta\alpha(\beta(e))$. So $te-\alpha(\beta(e))\in \text{Ker}(\beta)$. It follows that
\begin{center}
    $t'te=t'\alpha(\beta(e))=\alpha(t'\beta(e))\in \text{Im}(\alpha)$ 
\end{center} 
Hence $t'tE\subseteq \text{Im}(\alpha)$. Thus $\alpha$ is a $u$-$S$-epimorphism. Therefore, $\alpha$ is a $u$-$S$-isomorphism.
\end{proof}

The following example gives a $u$-$S$-injective $u$-$S$-envelope of a module $M$ that is not an injective envelope of $M$.

\begin{example}
  Let $R=\mathbb{Z}$, $S=\mathbb{Z}\setminus\{0\}$, and $E=\mathbb{Z}_{15}$. Then by Proposition \ref{prop1}(2) and since $E$ is a $u$-$S$-torsion $R$-module, we have $M=3\mathbb{Z}_{15}$ is a $u$-$S$-essential submodule of $E$ and so the inclusion map $i_{M}: M \to E$ is a $u$-$S$-essential $u$-$S$-monomorphism. Since $E$ is a $u$-$S$-torsion $R$-module, $E$ is $u$-$S$-injective by \cite[Corollary 4.4]{QK}. Thus by Theorem \ref{thm4}, $i_{M}: M \to E$ is a $u$-$S$-injective $u$-$S$-envelope of $M$. However, $i_{M}: M \to E$ is not an injective envelope of $M$ since $E$ is not an injective $R$-module.   
\end{example}

The following theorem shows that a finite direct sum of $u$-$S$-injective $u$-$S$-envelopes is a $u$-$S$-injective $u$-$S$-envelope.

\begin{thm}\label{thm5}
Let $S$ be a multiplicative subset of a ring $R$. Suppose that $f_i:M_i\to E_i$ is a $u$-$S$-injective $u$-$S$-envelope for each $i=1,2,\cdots,n$. Then $\bigoplus\limits_{i=1}^{n}f_i:\bigoplus\limits_{i=1}^{n}M_i\to \bigoplus\limits_{i=1}^{n}E_i$ is a $u$-$S$-injective $u$-$S$-envelope.   
\end{thm}

\begin{proof}
    Let $f:=\bigoplus\limits_{i=1}^{n}f_i$. Then $\text{Ker}(f)=\bigoplus\limits_{i=1}^{n}\text{Ker} (f_i)$. Since for each $i=1,2,\cdots,n$, $\text{Ker} (f_{i})$ is $u$-$S$-torsion, then for each $i=1,2,\cdots,n$, there exists $s_i\in S$ such that $s_i\text{Ker} (f_{i})=0$. Let $s:=s_1s_2\cdots s_n$. Then $s\text{Ker} (f)=0$. So $\text{Ker} (f)$ is $u$-$S$-torsion. That is, $f$ is a $u$-$S$-monomorphism. Also, since $E_i$ is $u$-$S$-injective for each $i=1,2,\cdots,n$, then $\bigoplus\limits_{i=1}^{n}E_i$ is $u$-$S$-injective by \cite[Proposition 4.7(1)]{QK}. Next, by hypothesis and Theorem \ref{thm4}, we have $\text{Im}(f_i)\unlhd^{u-S} E_i$ for each $i=1,2,\cdots,n$. So by Corollary \ref{cor1}, $\text{Im}(f)=\bigoplus\limits_{i=1}^{n}\text{Im}(f_i)\unlhd^{u-S} \bigoplus\limits_{i=1}^{n}E_i$. Again by Theorem \ref{thm4}, $f$ is a $u$-$S$-injective $u$-$S$-envelope.   
\end{proof}

 An arbitrary direct sum of $u$-$S$-injective $u$-$S$-envelopes need not be a $u$-$S$-injective $u$-$S$-envelope, as the following example shows. 

\begin{example}\label{exp7}
 Let $R= \mathbb{Z}$, $p$ a prime number, and $S = \{p^n \mid n=0,1,2,\cdots\}$. For each $n \geq 1$, let $M_n = \frac{\mathbb{Z}}{p^n\mathbb{Z}}$, $K_n=\frac{p\mathbb{Z}}{p^n\mathbb{Z}}$, and $f_n :M_n\to \frac{M_n}{K_n}$ be the natural map. Since $\text{Ker}(f_n)=K_n$ is $u$-$S$-torsion for each $n\geq 1$, each $f_n$ is a $u$-$S$-monomorphism. Also, since $p\in S$ and $p\big(\frac{\mathbb{Z}}{p\mathbb{Z}}\big)=0$, then for each $n \geq 1$, $\frac{M_n}{K_n}\cong \frac{\mathbb{Z}}{p\mathbb{Z}}$ is $u$-$S$-torsion, and so for each $n \geq 1$, $\frac{M_n}{K_n}$ is $u$-$S$-injective by \cite[Corollary 4.4]{QK}. Thus, by Proposition \ref{prop10}(1), each $f_n$ is a $u$-$S$-injective $u$-$S$-preenvelope. Moreover, we have $\text{Im}(f_n)=\frac{M_n}{K_n}\unlhd^{u-S}\frac{M_n}{K_n}$ for each $n\geq 1$. Thus, each $f_n$ is a $u$-$S$-injective $u$-$S$-envelope by Theorem \ref{thm4}. However, the map $f:=\bigoplus\limits_{n=1}^{\infty}f_n:\bigoplus\limits_{n=1}^{\infty}M_n\to \bigoplus\limits_{n=1}^{\infty} \frac{M_n}{K_n}$ is not a $u$-$S$-injective $u$-$S$-envelope since $\text{Ker}(f)=\bigoplus\limits_{n=1}^{\infty} \text{Ker}(f_n)=\bigoplus\limits_{n=1}^{\infty} K_n$ is not $u$-$S$-torsion, that is, $f$ is not a $u$-$S$-monomorphism.
\end{example}

 Let $S$ be a multiplicative subset of a ring $R$. Recall that $R$ is called $u$-$S$-Noetherian if there is an element $s\in S$ such that for any ideal $I$ of $R$, $sI\subseteq J$ for some finitely generated sub-ideal $J$ of $I$, and $S$ is called regular if $S\subseteq \text{reg}(R)$ \cite{QK}. The next result proves that any direct sum of injective envelopes is a $u$-$S$-injective $u$-$S$-preenvelope, provided that $R$ is $u$-$S$-Noetherian and $S$ is regular.
\begin{thm}\label{thm6}
    Let $R$ be a $u$-$S$-Noetherian ring and $S$ a regular multiplicative subset of $R$. Let $i_\alpha:M_\alpha\to E(M_\alpha)$, $\alpha\in \Delta$, be the injective envelopes. Then $\bigoplus\limits_{\alpha\in \Delta}i_\alpha:\bigoplus\limits_{\alpha\in \Delta} M_\alpha\to \bigoplus\limits_{\alpha\in \Delta}E(M_\alpha)$ is a $u$-$S$-injective $u$-$S$-preenvelope. 
\end{thm}

\begin{proof}
  Since each $i_\alpha$ is a monomorphism, $\bigoplus\limits_{\alpha\in \Delta}i_{\alpha}$ is a monomorphism \cite{AF}, and hence $\bigoplus\limits_{\alpha\in \Delta}i_{\alpha}$ is a $u$-$S$-monomorphism. Since $R$ is $u$-$S$-Noetherian and $E(M_\alpha)$ is injective for each $\alpha\in \Delta$, then by \cite[Theorem 4.10]{QK}, $\bigoplus\limits_{\alpha\in \Delta}E(M_\alpha)$ is $u$-$S$-injective. Thus, by Proposition \ref{prop10}(1), $\bigoplus\limits_{\alpha\in \Delta}i_\alpha$ is a $u$-$S$-injective $u$-$S$-preenvelope. 
\end{proof}

\begin{proposition}\label{prop11}
Let $S$ be a multiplicative subset of a ring $R$. Suppose that $f:M\to E$ is a $u$-$S$-injective $u$-$S$-envelope. Then the following statements hold.
\begin{enumerate}
    \item[(1)] $M$ is $u$-$S$-injective if and only if $M$ is $u$-$S$-isomorphic to $E$.
    \item[(2)] If $N\unlhd^{u-S}M$ and $g:N\to E'$ is a $u$-$S$-injective $u$-$S$-envelope of $N$, then $E$ is $u$-$S$-isomorphic to $E'$.
\end{enumerate}
\end{proposition}

\begin{proof} (1) Suppose that $M$ is $u$-$S$-injective. Since the identity map $1_M:M\to M$, and $f:M\to E$ are $u$-$S$-injective $u$-$S$-envelopes of $M$, then by Proposition \ref{prop8}, $M$ is $u$-$S$-isomorphic to $E$. The converse follows from \cite[Proposition 4.7 (3)]{QK} and the fact that $E$ is $u$-$S$-injective. \\[0.1cm]
(2) Since $N\unlhd^{u-S} M$, $i_{N}: N \to M$ is a $u$-$S$-essential monomorphism by Proposition \ref{prop7}. But $f:M\to E$ is a $u$-$S$-essential $u$-$S$-monomorphism by Theorem \ref{thm4}. Then $fi_N:N\to E$ is a $u$-$S$-essential $u$-$S$-monomorphism. Indeed, $fi_N$ is a $u$-$S$-monomorphism by the proof of \cite[Proposition 3.3]{ZQ}. Also, since $N\unlhd^{u-S} M$ and $f$ is a $u$-$S$-monomorphism, then $f(N)\unlhd^{u-S}f(M)$ by Proposition \ref{prop6}(2), but $f(M)\unlhd^{u-S} E$, so $\text{Im}(fi_N)=f(N)\unlhd^{u-S} E$ by Proposition \ref{prop5}(1). That is, $fi_N$ is $u$-$S$-essential. So we have $fi_N:N\to E$ is a $u$-$S$-essential $u$-$S$-monomorphism with $E$ $u$-$S$-injective. Hence $fi_N:N\to E$ is a $u$-$S$-injective $u$-$S$-envelope of $N$ by Theorem \ref{thm4}. But $g:N\to E'$ is a $u$-$S$-injective $u$-$S$-envelope of $N$, so $E$ is $u$-$S$-isomorphic to $E'$ by Proposition \ref{prop8}. 
\end{proof}

The following two lemmas are needed in the proof of Proposition \ref{prop12}.

\begin{lemma}\label{lem4} Let $S$ be a multiplicative subset of a ring $R$. A $u$-$S$-monomorphism $f:A\to B$ is $u$-$S$-essential if and only if for any $R$-homomorphism $h$, if $hf$ is a $u$-$S$-monomorphism, then $h$ is a $u$-$S$-monomorphism.
    \end{lemma}
\begin{proof}
    Let $f:A\to B$ be a $u$-$S$-monomorphism and $K=\text{Im}(f)$. Then $f':A\to K$ defined by $f'(x)=f(x)$, $x\in A$, is a $u$-$S$-isomorphism. We have $f=i_{K}f'$, where $i_{K}:K\to B$ is the inclusion map. By \cite[Lemma 2.1]{ZQ}, there is a $u$-$S$-isomorphism $g:K\to A$ such that $f'g=s1_{K}$ and $gf'=s1_A$ for some $s\in S$. Now, let $h$ be any $R$-homomorphism. Since 
    \begin{center}
        $hfg=hi_{K}f'g=hi_{K}s1_K=shi_{K}$~ and ~$shf=hfs1_A=hfgf'=shi_Kf'$, 
    \end{center}
    we have $g(\text{Ker}(hi_K))\subseteq \text{Ker}(hf)$ and $sf'(\text{Ker}(hf))\subseteq \text{Ker}(hi_K)$. Since $g$ and $sf'=s1_{K}f'$ are $u$-$S$-monomorphisms, then by Lemmas \ref{lem2} and \ref{lem3}(2), we have $\text{Ker}(hf)$ is $u$-$S$-torsion if and only if $\text{Ker}(hi_K)$ is $u$-$S$-torsion. That is, $hf$ is a $u$-$S$-monomorphism if and only if $hi_K$ is a $u$-$S$-monomorphism. By Proposition \ref{prop7}, the proof is complete.
\end{proof}

\begin{lemma}\label{lem5}
 Let $S$ be a multiplicative subset of a ring $R$. Suppose that the following diagram 
  \[
\xymatrix{
& A \ar[dl]_f \ar[dr]^g & \\
B  \ar[rr]_\varphi & & C
}
\]
is commutative with $f$ and $g$ are $u$-$S$-monomorphisms and $\varphi$ is a $u$-$S$-isomorphism. Then $f$ is $u$-$S$-essential if and only if $g$ is $u$-$S$-essential.
\end{lemma}

\begin{proof} First, since $\varphi$ is a $u$-$S$-isomorphism, there exist a $u$-$S$-isomorphism $\psi:C\to B$ and $t\in S$ such that $\psi \varphi =t1_{B}$ and $\varphi\psi =t1_{C}$ by \cite[Lemma 2.1]{ZQ}. Now, let $f$ be $u$-$S$-essential. Suppose that $hg$ is a $u$-$S$-monomorphism. Then $h\varphi f=hg$ is a $u$-$S$-monomorphism. So by Lemma \ref{lem4} and since $f$ is $u$-$S$-essential, we have $h\varphi$ is a $u$-$S$-monomorphism. Hence $th=ht1_C=(h\varphi)\psi$ is a $u$-$S$-monomorphism, being a composition of two $u$-$S$-monomorphisms. Since $\text{Ker}(h)\subseteq \text{Ker}(th)$, then $\text{Ker}(h)$ is $u$-$S$-torsion by Lemma \ref{lem2}. Thus $h$ is a $u$-$S$-monomorphism. Again by Lemma \ref{lem4}, $g$ is $u$-$S$-essential. The proof of the converse is similar.
\end{proof}

The following proposition gives a characterization of a $u$-$S$-injective $u$-$S$-envelope, when it exists.

\begin{proposition}\label{prop12}
      Let $S$ be a multiplicative subset of a ring $R$ and $M$ an $R$-module. Suppose that $M$ has a $u$-$S$-injective $u$-$S$-envelope. Then the following statements about a $u$-$S$-monomorphism $i:M\to E$ are equivalent:
     \begin{enumerate}
         \item[(1)] $i:M\to E$ is a $u$-$S$-injective $u$-$S$-envelope;
       
         \item[(2)] $E$ is $u$-$S$-injective and for any $u$-$S$-monomorphism $f:M\to Q$ with $Q$ $u$-$S$-injective, there is a $u$-$S$-monomorphism $g:E\to Q$ such that the following diagram
            \[\xymatrix{
& Q &\\
&M\ar[u]^{sf} \ar[r]^{i}&E\ar@{->}[lu]_{g} 
}\] commutes for some $s\in S$;
         \item[(3)] $i$ is a $u$-$S$-essential $u$-$S$-monomorphism and for any $u$-$S$-essential $u$-$S$-monomorphism $f:M\to N$, there is a $u$-$S$-monomorphism $g:N\to E$ such that the following diagram
            \[\xymatrix{
& E &\\
&M\ar[u]^{si} \ar[r]^{f}&N\ar@{->}[lu]_{g}
}\] commutes for some $s\in S$.
     \end{enumerate}
\end{proposition}

\begin{proof}
$(1)\Rightarrow (2)$: By $(1)$, $E$ is $u$-$S$-injective. Let $f:M\to Q$ be a $u$-$S$-monomorphism with $Q$ $u$-$S$-injective. Now, since $Q$ is $u$-$S$-injective, then there is an $R$-homomorphism $g:E\to Q$ such that $sf=gi$ for some $s\in S$. Since $gi=sf$ is a $u$-$S$-monomorphism and $i$ is $u$-$S$-essential, then by Lemma \ref{lem4}, $g$ is a $u$-$S$-monomorphism.\\
$(1)\Rightarrow (3)$: By $(1)$ and Theorem \ref{thm4}, $i$ is a $u$-$S$-essential $u$-$S$-monomorphism. Let $f:M\to N$ be a $u$-$S$-essential $u$-$S$-monomorphism. Since $E$ is $u$-$S$-injective, there is an $R$-homomorphism $g:N\to E$ such that $si=gf$ for some $s\in S$. Since $gf=si$ is a $u$-$S$-monomorphism and $f$ is $u$-$S$-essential, then by Lemma \ref{lem4}, $g$ is a $u$-$S$-monomorphism.\\
$(2)\Rightarrow (1)$: By (2), $E$ is $u$-$S$-injective. Let $f:M\to Q$ be a $u$-$S$-injective $u$-$S$-envelope of $M$. Then $f:M\to Q$ is a $u$-$S$-monomorphism with $Q$ $u$-$S$-injective. By $(2)$, there is a $u$-$S$-monomorphism $g:E\to Q$ such that $sf=gi$ for some $s\in S$. Note that $sf=(s1_Q)f$, $f$ is $u$-$S$-essential, and $s1_Q:Q\to Q$ is a $u$-$S$-isomorphism, so $sf$ is $u$-$S$-essential by Lemma \ref{lem5}. Now, since $E$ is $u$-$S$-injective, then by  \cite[Corollary 2.7(1)]{ZQ}, the $u$-$S$-exact sequence \begin{center}
    $0\to E\xrightarrow{g} Q\to \text{Coker}(g)\to 0$
\end{center} is $u$-$S$-split. Hence, there exist $t\in S$ and an $R$-homomorphism $g':Q \to E$ such that $g'g=t 1_{E}$. 

Let $y\in Q$. Then $g'(y)\in E$. So $tg'(y)=g'g(g'(y))$. This implies that $ty-g(g'(y))\in \text{Ker}(g')$ and so $ty\in \text{Im}(g)+\text{Ker}(g')$. Hence $tQ\subseteq \text{Im}(g)+\text{Ker}(g')$. Also, $t\big(\text{Im}(g)\cap \text{Ker}(g')\big)=0$ since if $g(x)\in \text{Ker}(g')$ where $x\in E$, then $tx=t 1_{E}(x)=g'g(x)=0$ and so $tg(x)=g(tx)=0$. Since $\text{Im}(sf)\subseteq  \text{Im}(g)$ and $sf$ is $u$-$S$-essential, so by Proposition \ref{prop5}(1), $\text{Im}(g)\unlhd^{u-S} Q$. So $s'\text{Ker}(g')=0$ for some $s'\in S$ and hence $s'tQ\subseteq \text{Im}(g)$. This means that $g$ is a $u$-$S$-epimorphism. Thus $g$ is a $u$-$S$-isomorphism. But $sf=gi$ and $sf$ is $u$-$S$-essential, so $i$ is $u$-$S$-essential by Lemma \ref{lem5}. Therefore, $(1)$ holds. \\
$(3)\Rightarrow (1)$: By (3), $i$ is a $u$-$S$-essential $u$-$S$-monomorphism. It remains to show that $E$ is $u$-$S$-injective. Let $f:M\to N$ be a $u$-$S$-injective $u$-$S$-envelope of $M$. By $(3)$,  there is a $u$-$S$-monomorphism $g:N\to E$ such that $si=gf$. Since $N$ is $u$-$S$-injective, the $u$-$S$-exact sequence \begin{center}
    $0\to N\xrightarrow{g} E\to \text{Coker}(g)\to 0$
\end{center} is $u$-$S$-split by  \cite[Corollary 2.7(1)]{ZQ}. By a similar argument as in the proof of the implication $(2)\Rightarrow (1)$, we obtain $g:N\to E$ is a $u$-$S$-isomorphism. But since $N$ is $u$-$S$-injective, we have $E$ is $u$-$S$-injective by \cite[Proposition 4.7(3)]{QK}.
\end{proof}

Let $R$ be a ring and $S$ a multiplicative subset of $R$. Define $$\mathcal{C}=\{M\mid M~\text{is an}~R\text{-module and}~E(M)~\text{is a prime}~R\text{-module}\}.$$ By Proposition \ref{prop2}, $M\unlhd^{u-S} E(M)$ for each $M\in \mathcal{C}$. Since $E(M)$ is $u$-$S$-injective, the inclusion map $M\hookrightarrow E(M)$ is a $u$-$S$-injective $u$-$S$-envelope of $M$ for each $M\in \mathcal{C}$. Hence, any $R$-module $M$ in $\mathcal{C}$ has a $u$-$S$-injective $u$-$S$-envelope. We end this paper with the following unanswered question:

\begin{question}
Let $R$ be a commutative ring and $S$ a multiplicative subset of $R$. Is it true that any $R$-module has a $u$-$S$-injective $u$-$S$-envelope?
\end{question}

\textbf{Conflict of interest:} The authors declare that they have no conflict of interest.

\end{document}